\documentclass[12pt]{amsart}
\usepackage{graphicx}
\usepackage{epstopdf}
\usepackage{graphicx}
\usepackage{amscd}
\usepackage{amsmath}
\usepackage{amsxtra}
\usepackage{amsfonts}
\usepackage{amssymb}

\usepackage{color}

\newtheorem{theorem}{Theorem}[section]

\newtheorem{lemma}[theorem]{Lemma}
\newtheorem{proposition}[theorem]{Proposition}

\theoremstyle{definition}
\newtheorem{definition}[theorem]{Definition}

\newtheorem{example}[theorem]{Example}
\theoremstyle{remark}

\renewcommand{\theclaim}{\textup{\theclaim}}

\newtheorem*{acknowledgements}{Acknowledgements}

\numberwithin{equation}{section}

\def\openone

{\mathchoice

{\hbox{\upshape \small1\kern-3.3pt\normalsize1}}

{\hbox{\upshape \small1\kern-3.3pt\normalsize1}}

{\hbox{\upshape \tiny1\kern-2.3pt\SMALL1}}

{\hbox{\upshape \Tiny1\kern-2pt\tiny1}}}

\makeatletter

\newbox\ipbox

\newcommand{\ip}[2]{\left\langle #1\, , \,#2\right\rangle}
\newcommand{\diracb}[1]{\left\langle #1\mathrel{\mathchoice

{\setbox\ipbox=\hbox{$\displaystyle \left\langle\mathstrut
#1\right.$}

\vrule height\ht\ipbox width0.25pt depth\dp\ipbox}

{\setbox\ipbox=\hbox{$\textstyle \left\langle\mathstrut
#1\right.$}

\vrule height\ht\ipbox width0.25pt depth\dp\ipbox}

{\setbox\ipbox=\hbox{$\scriptstyle \left\langle\mathstrut
#1\right.$}

\vrule height\ht\ipbox width0.25pt depth\dp\ipbox}

{\setbox\ipbox=\hbox{$\scriptscriptstyle \left\langle\mathstrut
#1\right.$}

\vrule height\ht\ipbox width0.25pt depth\dp\ipbox}

}\right. }

\newcommand{\dirack}[1]{\left. \mathrel{\mathchoice

{\setbox\ipbox=\hbox{$\displaystyle \left.\mathstrut
#1\right\rangle$}

\vrule height\ht\ipbox width0.25pt depth\dp\ipbox}

{\setbox\ipbox=\hbox{$\textstyle \left.\mathstrut
#1\right\rangle$}

\vrule height\ht\ipbox width0.25pt depth\dp\ipbox}

{\setbox\ipbox=\hbox{$\scriptstyle \left.\mathstrut
#1\right\rangle$}

\vrule height\ht\ipbox width0.25pt depth\dp\ipbox}

{\setbox\ipbox=\hbox{$\scriptscriptstyle \left.\mathstrut
#1\right\rangle$}

\vrule height\ht\ipbox width0.25pt depth\dp\ipbox}

} #1\right\rangle}

\newcommand{\cj}[1]{\overline{#1}}

\newcommand{\bz}{\mathbb{Z}}

\newcommand{\br}{\mathbb{R}}

\newcommand{\bn}{\mathbb{N}}

\def\blfootnote{\xdef\@thefnmark{}\@footnotetext}

\renewcommand{\mod}{\operatorname{mod}}

\input xy
\xyoption{all}
\usepackage{amssymb}

\def\-{^{-1}}

\def\ty{\emptyset}

\def\Z{\mathbb{Z}}

 \numberwithin{equation}{section}
 
\begin{document}

\title [Fourier bases and Fourier frames on self-affine measures] {Fourier bases and Fourier frames on self-affine measures}

\author{Dorin Ervin Dutkay}
\address{[Dorin Ervin Dutkay] University of Central Florida\\
	Department of Mathematics\\
	4000 Central Florida Blvd.\\
	P.O. Box 161364\\
	Orlando, FL 32816-1364\\
U.S.A.\\} \email{Dorin.Dutkay@ucf.edu}

\author{Chun-Kit Lai}

\address{Department of Mathematics, San Francisco State University,
1600 Holloway Avenue, San Francisco, CA 94132.}

 \email{cklai@sfsu.edu}

 \author{Yang Wang}

\address{Department of Mathematics, Hong Kong University of Science and Technology, Hong Kong}

 \email{yangwang@ust.hk}

\thanks{}

\begin{abstract}
This paper gives a review of the recent progress in the study of Fourier bases and Fourier frames on self-affine measures. In particular, we emphasize the new matrix analysis approach for checking the completeness of a mutually orthogonal set. This method helps us settle down a long-standing conjecture that Hadamard triples generates self-affine spectral measures. It also gives us non-trivial examples of fractal measures with Fourier frames. Furthermore, a new avenue is open to investigate whether the Middle Third Cantor measure admits Fourier frames.
\end{abstract}

\maketitle

\tableofcontents

\section{Introduction}

It is well known that the set of exponential functions $\{e^{2\pi in x}:~n\in\Z\}$ is an orthonormal basis for $L^2(I)$ where $I=[0,1]$. An interval is not the only set on which an orthogonal basis consisting of exponential functions exists for $L^2(\Omega)$. For example, for the set $\Omega=[0,1]\cup [2,3]$ the set of exponential functions $\{e^{2\pi i\lambda x}:~\lambda\in\Z \cup \Z+\frac{1}{4}\}$ is an orthogonal basis for $L^2(\Omega)$. It is an interesting question to ask what sets $\Omega$ admit an orthogonal basis consisting of exponential functions. In 1974 Fuglede, in a study \cite{Fug74} that is quite well known today, proposed the following infamous conjecture, widely known today as Fuglede's Conjecture:

\medskip

\noindent
{\bf Fuglede's Conjecture:}~~A measurable set $K$ is a spectral set in ${\mathbb R}^d$ if and only if $K$ tiles ${\mathbb R}^d$ by translation.

\medskip

Fuglede coined the term {\em spectral set} to denote sets $\Omega$ that admit admit an orthogonal basis of exponential functions. Fuglede's Conjecture has been studied by many investigators, including the authors of this paper, Jorgensen, Pedersen, Lagarias, {\L}aba, Kolountzakis, Matolcsi, Iosevich, Tao, and others (\cite{JP98,MR1700084,IKT1,IKT2,Ko1,MR2237932,La,LRW,LaWa96,LaWa97,Tao04}), but it had baffled experts for 30 years until Terence Tao \cite{Tao04} gave the first counterexample of a spectral set which is not a tile in ${\mathbb R}^d$, $d\geq5$. The example and technique were refined later to disprove the conjecture in both directions on ${\mathbb R}^d$ for $d\geq3$ \cite{MR2159781,MR2237932,MR2264214}. It has remained open in dimensions $d=1$ and $d=2$.

\medskip

Although the Fuglede's Conjecture in its original form has been disproved, there is a clear connection between spectral sets and tiling that has remained a mystery. Furthermore, spectral sets are apparently only one of the problems among an extremely broad class of problems involving complex exponential functions as either bases or more generally, frames.

\medskip

The rest of this section will focus on Fourier bases and frames for measures. Basic concepts and some of the important recent results will be introduced and reviewed. They will be the gateway to further discussions on the subject in later sections, which include open questions and more recent progresses.

\subsection{ Fourier bases and spectral measures.}

One extension of spectral sets concerns the more general spectral measures. A {\it spectral measure} is a bounded Borel measure $\mu$ such that there exists a set of complex exponentials  $E(\Lambda): =\{e_\lambda\}_{\lambda\in\Lambda}$, where $e_\lambda:=e^{2\pi i \langle \lambda,x\rangle}$, that forms an orthogonal basis of $L^2(\mu)$. If such $\Lambda$ exists, it is called a \textit{spectrum} for $\mu$. A spectral set can be viewed as a special case of spectral measure by considering the Lebesgue measure $\chi_K dx$ on $K$. Of particular intrigue are measures that are singular, especially self-similar measures that are closely tied to the field of fractal geometry and self-affine tiles.

The study of Fourier bases on singular measures started with the Jorgensen-Pedersen paper \cite{JP98} in which they proved the following surprising result: consider the construction of a Cantor set with scale 4: take the unit interval, divide it into 4 equal pieces, keep the first and the third piece. Then for the two remaining intervals, perform the same procedure and keep repeating the procedure for the remaining intervals ad infinitum. The remaining set is the Cantor-4 set in the Jorgensen-Pedersen example. On this set consider the appropriate Hausdorff measure of dimension $\frac12$, which we denote by $\mu_4$.

Jorgensen and Pedersen proved that the Hilbert space $L^2(\mu_4)$ has an orthonormal family of exponential functions, i.e., a Fourier basis, namely
$$
   \left\{e_\lambda : \lambda=\sum_{k=0}^n4^kl_k, l_k\in\{0,1\}, n\in\bn\right\}.
$$
Moreover, they also proved that the Middle-Third-Cantor set, with its corresponding Hausdorff measure cannot have more than two mutually orthogonal exponential functions.

\begin{definition}\label{defsp}
Let $\mu$ be a compactly supported Borel probability measure on
${\br}^d$ and $\langle\cdot,\cdot\rangle$ denote the standard inner product. We say that $\mu$ is a {\it spectral measure} if there
exists a countable set $\Lambda\subset {\mathbb R}^d$, which we call a {\it spectrum}, such  that
$E(\Lambda): = \{e^{2\pi i \langle\lambda,x\rangle}:
\lambda\in\Lambda\}$ is an orthonormal basis for $L^2(\mu)$. The Fourier transform of $\mu$ is defined to be
 $$
\widehat{\mu}(\xi)= \int e^{-2\pi i \langle\xi,x\rangle}d\mu(x).
 $$
 It is easy to verify that a measure is a spectral measure with spectrum $\Lambda$ if and only if the following two conditions are satisfied:
 \begin{enumerate}
\item (Orthogonality) $ \widehat{\mu}(\lambda-\lambda')=0$ for all distinct $\lambda,\lambda'\in\Lambda$ and
\item (Completeness) If for $f\in L^2(\mu)$, $\int f(x)e^{-2\pi i \langle\lambda,x\rangle}d\mu(x)=0$ for all $\lambda\in\Lambda$, then  $f=0$.
    \end{enumerate}
\end{definition}

The Jorgensen-Pedersen Cantor-4 set and the measure $\mu_4$ can be seen as the attractor and invariant measure of the iterated function system
$$\tau_0(x)=\frac{x}{4},\quad \tau_2(x)=\frac{x+2}{4}.$$

\begin{definition}\label{defifs}
For a given expansive $d\times d$ integer matrix $R$ and a finite set of integer vectors $B$ in $\bz^d$, with $\#B =: N$, we define the {\it affine iterated function system} (IFS) $\tau_b(x) = R^{-1}(x+b)$, $x\in \br^d, b\in B$. The {\it self-affine measure}, also called the {\it invariant measure}, (with equal weights) is the unique probability measure $\mu = \mu(R,B)$ satisfying
\begin{equation}\label{self-affine}
\mu(E) = \sum_{b\in B} \frac1N \mu (\tau_b^{-1} (E)),\mbox{ for all Borel subsets $E$ of $\br^d$.}
\end{equation}
This measure is supported on the {\it attractor} $T(R,B)$ which is the unique compact set that satisfies
$$
T(R,B)= \bigcup_{b\in B} \tau_b(T(R,B)).
$$
The set $T(R,B)$ is also called the {\it self-affine set} associated with the IFS. See \cite{Hut81} for details.
\end{definition}

The starting idea in Jorgensen-Pedersen's construction is to find first orthogonal exponential functions at a finite level and then iterate by rescaling and translations. If we start with the atomic measure
$\delta_{\frac14 B}:=\delta_{\frac14\{0,2\}}=\frac{1}{2}(\delta_0+\delta_2)$, we find that $\{e_{\ell} : \ell\in L:=\{0,1\}\}$ is an orthonormal basis for $L^2(\delta_{\frac14 B})$. In other words, $(R=4, B=\{0,2\}, L=\{0,1\})$ form a Hadamard triple.

\begin{definition}\label{hada}
Let $R\in M_d({\mathbb Z})$ be an $d\times d$ expansive matrix (all eigenvalues have modulus strictly greater than 1) with integer entries. Let $B, L\subset{\mathbb Z}^d $ be a finite set of integer vectors with $N= \#B=\#L$ ($\#$ means the cardinality). We assume without loss of generality that $0\in B$ and $0\in L$. We say that the system $(R,B,L)$ forms a {\it Hadamard triple} (or $(R^{-1}B, L)$ forms a {\it compatible pair} in \cite{LaWa02} ) if the matrix
\begin{equation}\label{Hadamard triples}
H=\frac{1}{\sqrt{N}}\left[e^{2\pi i \langle R^{-1}b,\ell\rangle}\right]_{\ell\in L, b\in B}
\end{equation}
is unitary, i.e., $H^*H = I$.

\end{definition}

The system $(R,B,L)$ forms a  Hadamard triple if and only if the Dirac measure $\delta_{R^{-1}B} = \frac{1}{\#B}\sum_{b\in B}\delta_{R^{-1}b}$ is a spectral measure with  spectrum $L$. Moreover, this property is a key property in producing examples of singular spectral measures, in particular spectral self-affine measures.

\medskip
Denote by
$$
B_n = B+RB+\dots+R^{n-1}B, \ \Lambda_n:=L+R^TL+\dots+(R^T)^{n-1}L.
$$
If $(R,B,L)$ form a Hadamard triple, then a simple computation (see e.g. \cite{JP98}) shows that $(R^n, B_n, \Lambda_n)$ is also a Hadamard triple, for all $n\in\bn$. In other words, the measures
$$\nu_n = \delta_{R^{-n}(B_n)}=\delta_{R^{-1}B}\ast\delta_{R^{-2}B}
\ast...\ast\delta_{R^{-n}B}$$
have a spectrum $\Lambda_n$
These measures approximate the singular measure $\mu(R,B)$ in the following sense,
\begin{equation}\label{conv prod}
\mu(R,B) = \delta_{R^{-1}B}\ast\delta_{R^{-2}B}
\ast\delta_{R^{-3}B}... = \mu_n\ast\mu_{>n}
\end{equation}
where, by self-similarity, the measure $\mu_{>n}(\cdot) = \mu(R^n\cdot) $. One should expect that, under the right conditions, the set $\Lambda = \bigcup_{n=1}^{\infty}\Lambda_n$ forms a spectrum for the measure $\mu(R,B)$ and this is the case for the Jorgensen-Pedersen example. However, this is not true in general, even though it always yields an orthonormal set, but, in some cases, this set can be incomplete. Here is a simple counterexample: consider the scale $R=2$, and the digits $B=\{0,1\}$. It is easy to see that $\mu(R,B)$ is the Lebesgue measure on $[0,1]$. One can pick $L=\{0,1\}$ to make a Hadamard triple. But when we construct the set $\Lambda$ we notice that it is actually $\Lambda=\bn\cup\{0\}$ and we know that the classical Fourier series are indexed by all the integers $\bz$, so the exponential functions with frequencies in $\Lambda$ are orthonormal, but not complete.

\medskip

The natural conjecture was raised to see {\it if for any Hadamard triple $(R,B,L)$ the measure $\mu(R,B)$ is spectral}. This conjecture  was settled on ${\mathbb R}^1$ \cite{LaWa02,DJ06}. The situation becomes more complicated on high dimension. Dutkay and Jorgensen showed that the conjecture  is true if $(R,B,L)$ satisfies a technical condition called {\it reducibility condition} \cite{DJ07d}. There are some other additional assumptions proposed by Strichartz guaranteeing the conjecture is true \cite{Str98,Str00}. Some low-dimensional special cases were also considered by Li \cite{MR3163581,MR3302160}. We eventually proved this conjecture in \cite{DL15,DHL15}.

\begin{theorem}\label{th1}
If $(R,B,L)$ is a Hadamard triple, then the self-affine measure $\mu=\mu(R,B)$ is a spectral measure.
\end{theorem}

The proof of this result highlights some of the useful techniques for the study of spectral self-affine measures. We will give a sketch of proof of this result as part of this review.

\bigskip

\subsection{  Fourier frames} A natural generalization of orthonormal basis is the notion of frames. It allows expansion of functions in a non-unique way, but is robust to perturbation of frequencies \cite{DHSW11}. Recall that a {\it frame} is a family of vectors $\{e_i: i \in I\}$ in a Hilbert space $H$ with the property that there exist constants $A,B>0$ (called {\it the frame bounds}) such that
\begin{equation}
A\|f\|^2\leq \sum_{i\in I}|\ip{f}{e_i}|^2\leq B\|f\|^2,\quad(f\in H).
\label{eqfr}
\end{equation}
A Borel measure $\mu$ is called a {\it frame-spectral} measure if there exists a family of exponential functions $\{e_\lambda : \lambda\in\Lambda\}$ forming a frame for $L^2(\mu)$.
\medskip

Soon after Jorgensen and Pedersen showed in \cite{JP98} that the Middle-Third-Cantor measure $\mu_3$ is not a spectral measure, a natural question was proposed by Strichartz \cite{Str00} who asked   whether the measure $\mu_3$  is frame-spectral. The question is still open.

\medskip

We proved in \cite{DL14} that if a measure has a Fourier frame, then it must have a certain homogeneity under local translations (so it must look the same locally at every point). This condition excludes the possibility of Fourier frames on self-affine measures with unequal weights, but not for the Middle-Third-Cantor measures. Some fairy large Bessel sequences of exponential functions (i.e., only the upper bound holds in \eqref{eqfr}) were constructed in \cite{MR2826404} for the Middle-Third-Cantor set. Some weighted Fourier frames were built by Picioroaga and Weber for the Cantor-4 set in \cite{PiWe15}.

\medskip

The following condition generalizing Hadamard triples was introduced in \cite{DL15} and \cite{LaiW}. We state the definition on ${\mathbb R}^1$, but it can be defined on any dimension.

\medskip

\begin{definition}\label{APFT}
 Let $\epsilon_j$ be such that $0\le\epsilon_j<1$ and $\sum_{j=1}^{\infty}\epsilon_j<\infty$. We say that $\{(N_j,B_j)\}$ is an {\it almost-Parseval-frame tower} associated to $\{\epsilon_j\}$ if
\begin{enumerate}
\item $N_j$ are integers and $N_j\geq2$ for all $j$;
\item  $B_j\subset \{0,1,...,N_j-1\}$ and $ 0 \in B_j$ for all $j$;
\item  Let $M_j := \#B_j$. There exists $L_j\subset {\mathbb Z}$ (with $0\in L_j$) such that for all $j$,
 \begin{equation}\label{ALmost}
(1-\epsilon_j)^2\sum_{b\in B_j}|w_b|^2\leq \sum_{\lambda\in L_j}\left|\frac{1}{\sqrt{M_j}}\sum_{b\in B_j}w_be^{-2\pi i b \lambda/N_j}\right|^2\leq (1+\epsilon_j)^2\sum_{b\in B_j}|w_b|^2
\end{equation}
 for all ${\bf w} = (w_b)_{b\in{B_j}}\in{\mathbb C}^{M_j}$.
 Letting the matrix ${\mathcal F}_j = \frac{1}{\sqrt{M_j}}\left[e^{2\pi i b \lambda/N_j}\right]_{\lambda\in L_j,b\in B_j}$ and $\|\cdot\|$ the standard Euclidean norm, (\ref{ALmost}) is equivalent to
\begin{equation}\label{ALmost1}
(1-\epsilon_j)\|{\bf w}\|\leq \|{\mathcal F}_j{\bf w}\|\leq (1+\epsilon_j)\|{\bf w}\|
\end{equation} for all  ${\bf w}\in {\mathbb C}^{M_j}$.
 \end{enumerate}
  Whenever  $\{L_j\}_{j\in{\mathbb Z}}$ exists, we call $\{L_j\}_{j\in{\mathbb Z}}$ a {\it pre-spectrum} for the almost-Parseval-frame tower.
We define the following measures associated to an almost-Parseval-frame tower.
$$
\nu_j = \frac{1}{M_{j}}\sum_{b\in B_{j}}\delta_{b/N_1N_2...N_j}
$$
and
\begin{equation}\label{measure}
\mu = \nu_1\ast\nu_2\ast.... : = \mu_n\ast\mu_{>n}
\end{equation}
where $\mu_n$ is the convolution of the first $n$ discrete measures and $\mu_{>n}$ is the remaining part.

\end{definition}

\medskip

When all $\epsilon_j=0$, $(N_j,B_j,L_j)$ forms a Hadamard triple.   We note that this class of measures includes self-similar measures because if we are given an integer $N\ge2$ and a set $B\subset\{0,1,...,N-1\}$ such that
$$
N_j = N^{n_j}, B_j= B+NB+...+N^{n_j-1}B,
$$
 then the associated measure $\mu$ is the self-affine measure $\mu(N,B)$. In particular if $N=3$ and $B = \{0,2\}$, $\mu$ is the standard Middle Third Cantor measure. In such situation, the almost-Parseval-frame tower is called {\it self-similar}. We prove that

 \begin{theorem} \label{th2}

 (i) If the self-similar almost-Parseval-frame tower as in Definition \ref{APFT} exists, then the associated self-similar measure is frame-spectral.

 \medskip

 (ii)  There exists almost-Parseval-frame tower with $\epsilon_j>0$ and the associated fractal measure is frame-spectral but not spectral.
 \end{theorem}

In the rest of the paper we go into more details on spectral self-affine measures as well as frame-spectral measures. We consider the explicit construction of self-affine frame-spectral measures that are not spectral. It is our hope that this survey summarizes not only recent results on the subject, but also some of the key techniques used to tackle problems. The open questions we discuss here should serve to quickly lead people into this area.

\medskip

\section{Spectral Self-Affine Measures}

 We begin with the following definitions.

\begin{definition}
We call a finite set $B\subset\bz^d$, a {\it simple digit set for $R$}, if distinct elements of $B$ are not congruent $(\mod R(\bz^d))$. We define $\bz[R,B]$ to be the smallest lattice in $\br^d$ that contains the set $B$ and is invariant under $R$, i.e., $R(\bz[R,B])\subset \bz[R,B]$.

\medskip

We say that $\mu(R,B)$ satisfies the {\it no-overlap} condition if
$$
\mu(\tau_b(T(R,B))\cap \tau_{b'}(T(R,B))) = 0, \ \forall b\neq b'\in B.
$$
\end{definition}

We have the following preliminary reductions that we can do to solve the problem.

\begin{enumerate}
\item If $(R,B,L)$ is a Hadamard triple, then $B$ is a simple digit set for $R$ and $L$ is a simple digit set for $R^T$.
\item We can assume without loss of generality that $\bz[R,B] = \bz^d$.
\item If $B$ is a simple digit for $R$, then $\mu(R,B)$ satisfies the no-overlap condition.
\end{enumerate}

Property (i) is a simple consequence of mutually orthogonality. If $\bz[R,B] \ne \bz^d$, we can conjugate some matrix to produce another Hadamard triple which satisfies with the desired property \cite[Proposition 4.1]{DL15}. The proof of the no-overlap condition of the self-affine measure can be referred to \cite[Section 2]{DL15}.

\medskip

The following proposition perhaps gives us the main idea on how to prove the completeness of an orthogonal set of exponentials and the proof are readily generalized to give our various results.

\begin{proposition}\label{prop_main}
Suppose that $(R,B,L)$ is a Hadadmard triple and $
\Lambda = \bigcup_{n=1}^{\infty}\Lambda_n $ where $\Lambda_n = L+R^TL+...+(R^T)^{n-1}L.$ Assume that
$$
\delta(\Lambda): = \inf_{n\ge1}\inf_{\lambda\in\Lambda_n}|\widehat{\mu}((R^T)^{-n}\lambda)|^2>0
$$
Then $\Lambda$ is a spectrum for $L^2(\mu(R,B))$.

\end{proposition}

\begin{proof}
The proof of mutual orthogonality follows directly from the fact that
$$
H_n = \frac{1}{\sqrt{N^n}}\left(e^{-2\pi i \ip{{ R}^{-n}\bf b}{\lambda}}\right)_{\lambda\in\Lambda_n,{\bf b}\in {B}_n}.
$$
is a unitary matrix. We now show the completeness by showing that the following frame bounds hold: for any $f\in L^2(\mu)$,
\begin{equation}\label{eq2.1i}
\delta(\Lambda)\|f\|^2\leq \sum_{\lambda\in\Lambda}\left|\int f(x)e^{-2\pi i \langle\lambda,x\rangle}d\mu(x)\right|^2\leq \|f\|^2.
\end{equation}
The positive lower bound implies the completeness. To prove (\ref{eq2.1i}), we just need to check it for a dense set of functions in $L^2(\mu)$. Let ${\bf 1}_E$ be the indicator function of the set $E$ and
$$
{\mathcal S}_n = \left\{\sum_{{\bf b}\in {B}_n}w_{\bf b}{\bf 1}_{\tau_{\bf b}(R,B)}: w_{\bf b}\in{\mathbb C}\right\}.
$$
Thus, ${\mathcal S}_n$ denotes the collection of all $n^{th}$ level step functions on $T(R,B)$. It forms an increasing union of sets. Let also ${\mathcal S} = \bigcup_{n=1}^{\infty}{\mathcal S}_n$. This ${\mathcal S}$ forms a dense set in  $L^2(\mu)$. Now for any $f = \sum_{{\bf b}\in {B}_n}w_{\bf b}{\bf 1}_{\tau_{\bf b}(R,B)}\in {\mathcal S}$, a direct computation shows that
\begin{equation}\label{eq3.1}
\int|f|^2d\mu = \frac{1}{N^n}\sum_{{\bf b}\in {B}_n}|w_{\bf b}|^2 =  \frac{1}{N^n}\|{\bf w}\|^2
\end{equation}
where ${\bf w} = (w_{\bf b})_{{\bf b}\in B_n}$ and
\begin{equation}\label{eq3.2}
\int f(x)e^{-2\pi i \ip{\lambda}{x}}d\mu(x) = \frac{1}{N^n}\widehat{\mu}((R^T)^{-n}\lambda)\sum_{{\bf b}\in {B}_n}w_{\bf b} e^{-2\pi i  \ip{R^{-n}{\bf b}}{ \lambda}}
\end{equation}
which means that
\begin{equation}\label{eq3.3}
\sum_{\lambda\in\Lambda_n}\left|\int f(x)e^{-2\pi i \ip{\lambda}{x}}d\mu(x)\right|^2 = \frac{1}{N^n} \sum_{\lambda\in\Lambda_n}|\widehat{\mu}((R^T)^{-n}\lambda)|^2\left|\sum_{{\bf b}\in {B}_n}\frac{1}{\sqrt{N^n}}w_{\bf b}e^{-2\pi i  \ip{R^{-n}{\bf b}}{\lambda}}\right|^2
\end{equation}
As $\delta(\Lambda)\le|\widehat{\mu}((R^T)^{-n}\lambda)|^2\le 1$, we obtain
$$
 \frac{1}{N^n}\delta(\Lambda)\|H_n{\bf w}\|^2\le \sum_{\lambda\in\Lambda_n}\left|\int f(x)e^{-2\pi i \ip{\lambda}{x}}d\mu(x)\right|^2 \le  \frac{1}{N^n}\|H_n{\bf w}\|^2 .
$$ But $H_n$ is a Hadamard matrix, so we have $\|H_n{\bf w}\| = \|{\bf w}\|$ and hence
$$
\delta(\Lambda)\int|f|^2d\mu\le \sum_{\lambda\in\Lambda_n}\left|\int f(x)e^{-2\pi i \ip{\lambda}{x}}d\mu(x)\right|^2 \le \int|f|^2d\mu .
$$
As ${\mathcal S}_n\subset{\mathcal S}_m$ for any $n<m$, we have
$$
\delta(\Lambda)\int|f|^2d\mu\le \sum_{\lambda\in\Lambda_m}\left|\int f(x)e^{-2\pi i \ip{\lambda}{x}}d\mu(x)\right|^2 \le \int|f|^2d\mu .
$$
By taking $m$ to infinity, we have (\ref{eq2.1i}).
\end{proof}

As we have discussed in the introduction, we cannot expect the standard orthogonal set to be always a spectrum. We have to look for some other alternatives. The main important observation is to distinguish two cases depending on whether the periodic zero set
\begin{equation}
\mathcal Z:=\{\xi\in\br^d : \widehat\mu(\xi+k)=0,\mbox{ for all }k\in\bz^d\}
\label{eqz}
\end{equation}
of $\mu(R,B)$ is empty or not.


For the case $\mathcal Z=\ty$, we will see that Theorem \ref{th1} follows along the same lines as Proposition \ref{prop_main}.  The case $\mathcal Z\neq\ty$ is more subtle but it can appear in higher dimensions $d$, but, in this case, the system $(R,B,L)$ has a special quasi-product form structure.

\subsection{The case $\mathcal Z=\ty$}

Given a sequence of integers $n_k$ and let $m_k = n_1+...+n_k$. The self-affine measure can be rewritten as
$$
\mu(R,B) = \delta_{R^{-n_1}B_{n_1}}\ast\delta_{R^{-m_2}B_{n_2}}\ast...\ast\delta_{R^{-m_k}B_{n_k}}\ast...
$$
Then note that if we have another set $J_{n_k}$ of integer vectors, with $J_{n_k}\equiv L_{n_k}$ (mod $(R^T)^{n_k}({\mathbb Z}^d)$), then it is easy to see that $(R^{n_k}, B_{n_k},J_{n_k})$ still form Hadamard triples. In this sequel, we produce many other mutually orthogonal sets, by:
\begin{equation}\label{eqLambda_k}
\Lambda_k = J_{n_1}+ (R^T)^{m_1}J_{n_2}+ (R^T)^{m_2} J_{n_3}+...+ (R^T)^{m_{k-1}} J_{n_{k}},
\end{equation}
\begin{equation}\label{eqlambda}
\ \Lambda = \bigcup_{k=1}^{\infty}\Lambda_k.
\end{equation}
Repeating the argument in Proposition \ref{prop_main}, the following holds true.

\begin{proposition}\label{prop_main1}
Suppose that $(R,B,L)$ is a Hadadmard triple and $
\Lambda$ if given in (\ref{eqlambda}). Assume that
\begin{equation}\label{eqdelta}
\delta(\Lambda): = \inf_{k\ge1}\inf_{\lambda\in\Lambda_k}|\widehat{\mu}((R^T)^{-m_k}\lambda)|^2>0
\end{equation}
Then $\Lambda$ is a spectrum for $L^2(\mu(R,B))$.
\end{proposition}

\medskip

The following proposition guarantees some $\Lambda$ will satisfy $\delta(\Lambda)>0$.

\begin{proposition}\label{prop_main2}
Suppose that ${\mathcal Z} = \emptyset$. Then there exists $\Lambda$ built as in \eqref{eqLambda_k} and \eqref{eqlambda} such that $\delta(\Lambda)>0$.
\end{proposition}

We now give the proof of this proposition. We start with a lemma.

\begin{lemma}\label{lem2.1}
Suppose that ${\mathcal Z} = \emptyset$ and let $X$ be any compact set on ${\mathbb R}^d$. Then there exist $\epsilon_0>0$, $\delta_0>0$ such that for all $x\in X$, there exists $k_x\in{\mathbb Z}^d$ such that for all $y\in\br^d$ with $\|y\|<\epsilon_0$, we have $|\widehat\mu(x+y+k_x)|^2\geq \delta_0$. In addition, we can choose $k_0=0$ if $0\in X$.
\end{lemma}

\begin{proof}
As ${\mathcal Z} = \emptyset$, for all $x\in X$ there exists $k_x\in\Gamma$ such that $\widehat\mu(x+k_x)\neq 0$. Since $\widehat\mu$ is continuous, there exists an open ball $B(x,\epsilon_x)$ and $\delta_x>0$ such that $|\widehat\mu(y+k_x)|^2\geq\delta_x$ for all $y\in B(x,\epsilon_x)$. Since $X$ is compact, there exist $x_1,\dots,x_m\in X$ such that
$$X\subset\bigcup_{i=1}^mB(x_i,\frac{\epsilon_{x_i}}2).$$
Let $\delta:=\min_i\delta_{x_i}$ and $\epsilon:=\min_i\frac{\epsilon_{x_i}}{2}$. Then, for any $x\in X$, there exists $i$ such that $x\in B(x_i,\frac{\epsilon_{x_i}}2)$. If $\|y\|<\epsilon$, then $x+y\in B(x_i,\epsilon_{x_i})$, so $|\widehat \mu(x+y+k_{x_i})|^2\geq \delta$, we can redefine $k_x$ to be $k_{x_i}$ to obtain the conclusion. Clearly, we can choose $k_0=0$ if $0\in X$ since $\widehat{\mu}(0)=1$.
\end{proof}

\medskip

\noindent{\it Proof of Proposition \ref{prop_main2}.}  Suppose that $(R,B,L)$ is a Hadamard triple $(R,B,L)$. Then we take $X = T(R^T,L)$, the self-affine set generated by $R^T$ and digit set $L$. Define $J_{n} = L+R^{T}L+...+(R^T)^{n-1}L$  
 By the definition of self-affine sets, 
$$
(R^{T})^{-(n+p)}J_{n}\subset X , \quad(n\in{\mathbb N},p\geq 0).
$$

Fix the $\epsilon_0$ and $\delta_0$ in Lemma \ref{lem2.1}. We now construct the sets $\Lambda_k$ and $\Lambda$ as in \eqref{eqLambda_k} and \eqref{eqlambda}, by replacing the sets $J_{n_k}$ by some sets $\widehat J_{n_k}$ to guarantee that the number $\delta(\Lambda)$ in \eqref{eqdelta} is positive.

 We first start with $\Lambda_{0}: = \{0\} $ and $m_0=n_0=0$. Assuming that $\Lambda_k$ has been constructed, we first choose our $n_{k+1}>n_k$ so that
\begin{equation}\label{eq4.5}
\|(R^T)^{-(n_{k+1}+p)}\lambda\|<\epsilon_0 , \ \forall \ \lambda\in\Lambda_k, p\geq0.
\end{equation}
We then define $m_{k+1} = m_k+n_{k+1}$ and
$$
\Lambda_{k+1} = \Lambda_k+(R^T)^{m_{k}}\widehat{J_{n_{k+1}}}
$$
where
$$
\widehat{J_{n_{k+1}}} = \{j+(R^T)^{n_{k+1}}k(j): j\in J_{n_{k+1}}, \ k(j)\in{\mathbb Z}^d  \}
$$
where $k(j)$ is chosen to be $k_x$ from Lemma \ref{lem2.1}, with $x = (R^{T})^{-n_{k+1}}j  \in X$. As $0\in J_{n_{k}}$ and $k_0=0$ for all $k$, the sets $\Lambda_k$ are of the form (\ref{eqLambda_k}) and  form an increasing sequence. For these sets $\Lambda_k$, we claim that the associated $\Lambda$ in $(\ref{eqlambda})$ satisfies $\delta(\Lambda)>0$.

\medskip

To justify the claim, we note that if $\lambda\in \Lambda_{k}$, then
$$
\lambda = \lambda'+(R^T)^{m_{k-1}}j  +(R^T)^{m_{k}}k(j),
$$
where $\lambda'\in \Lambda_{k-1}$, $j\in J_{n_{k}}$. This means that
$$
(R^T)^{-m_k}\lambda = (R^T)^{-m_k}\lambda'+(R^T)^{-n_k}j+ k(j).
$$
By (\ref{eq4.5}), $\|(R^T)^{-m_k}\lambda'\|<\epsilon_0$. From Lemma \ref{lem2.1}, since $(R^T)^{-n_k}j\in X$, we must have $|\widehat{\mu}((R^T)^{-m_k}\lambda)|^2\geq\delta_0>0$. As $\delta_0$ is independent of $k$, the claim is justified and hence this completes the proof of the proposition.
\qquad$\Box$.

Combining Proposition \ref{prop_main1} and Proposition \ref{prop_main2}, we settle the case ${\mathcal Z}=\emptyset$.

\begin{theorem}\label{th_Z}
Suppose that ${\mathcal Z} = \emptyset$ and $(R,B,L)$ is a Hadamard triple. Then the self-affine measure $\mu(R,B)$ is a spectral measure.
\end{theorem}

%
%
%

\medskip

\subsection{The case $\mathcal Z\neq \ty$} When $\mathcal Z\neq \ty$, it means that there is an exponential function $e_{\xi}$ such that it is orthogonal to every exponential function with integer frequencies. This implies that none of the subsets of integers can be complete. Therefore, any construction of orthogonal sets within integers must fail to be spectral. We illustrate the situation through a simple example.

\begin{example}\label{example5.3}
Let $R = \left[
           \begin{array}{cc}
             4 & 0 \\
             1 & 2 \\
           \end{array}
         \right]
$, $$B= \left\{\left[
                \begin{array}{c}
                  0 \\
                  0 \\
                \end{array}
              \right], \left[
                \begin{array}{c}
                  0 \\
                  3\\
                \end{array}
              \right], \left[
                \begin{array}{c}
                  1 \\
                  0 \\
                \end{array}
              \right], \left[
                \begin{array}{c}
                  1 \\
                  3 \\
                \end{array}
              \right]
\right\} \ \mbox{and} \
 L= \left\{\left[
                \begin{array}{c}
                  0 \\
                  0 \\
                \end{array}
              \right], \left[
                \begin{array}{c}
                  2 \\
                  0\\
                \end{array}
              \right],  \left[
                \begin{array}{c}
                  0 \\
                  1 \\
                \end{array}
              \right], \left[
                \begin{array}{c}
                  2 \\
                  1 \\
                \end{array}
              \right]
\right\}.$$ Then $(R,B,L)$ forms a  Hadamard triple and ${\mathbb Z}[R,B] = {\mathbb Z}^2$. However, the set  defined in \eqref{eqz} ${\mathcal Z}\neq \emptyset$ for the measure $\mu = \mu(R,B)$.
\end{example}

\begin{proof}
It is a direct check to see $(R,B,L)$ forms a  Hadamard triple and ${\mathbb Z}[R,B] = {\mathbb Z}^2$.  As $M_B(\xi_1,\xi_2) = \frac{1}{4}(1+e^{2\pi i \xi_1})(1+e^{2\pi i 3\xi_2})$. It follows that the zero set of $M_B$, denoted by $Z(M_B)$, is equal to
$$
Z(M_B) = \left\{\left[
                \begin{array}{c}
                 \frac12+n \\
                  y \\
                \end{array}
              \right]: n\in{\mathbb Z}, y\in {\mathbb R}\right\}\cup \left\{\left[
                \begin{array}{c}
                  x \\
                  \frac16+\frac13 n  \\
                \end{array}
              \right]:x\in {\mathbb R}, n\in{\mathbb Z}\right\}.
$$
Let $(R^T)^j = \left[
           \begin{array}{cc}
             4^j & a_j \\
              0& 2^j \\
           \end{array}
         \right]$, for some $a_j\in{\mathbb Z}$. As $\widehat{\mu}(\xi) = \prod_{j=1}^{\infty}M_B((R^T)^{-j}(\xi))$, the zero set of $\widehat{\mu}$, denoted by $Z(\widehat{\mu})$, is equal to
$$\begin{aligned}
Z(\widehat{\mu}) =& \bigcup_{j=1}^{\infty}(R^{T})^{j} Z(M_B) \nonumber\\
=&\bigcup_{j=1}^{\infty}\left\{\left[
                \begin{array}{c}
                 4^j(\frac12+n)+a_jy \\
                  2^jy \\
                \end{array}
              \right]: n\in{\mathbb Z}, y\in {\mathbb R}\right\}\cup \left\{\left[
                \begin{array}{c}
                  4^jx+a_j(\frac16+\frac13 n) \\
                  2^{j}\left(\frac16+\frac13 n\right)  \\
                \end{array}
              \right]:x\in {\mathbb R}, n\in{\mathbb Z}\right\}.
\end{aligned}
$$
We claim that the points in  $\left[
                \begin{array}{c}
                 0 \\
                  \frac13 \\
                \end{array}
              \right]+{\mathbb Z}^2$ are in $Z(\widehat{\mu})$ which shows ${\mathcal Z}\neq \emptyset$. Indeed, for any $\left[
                \begin{array}{c}
                 m \\
                  \frac13+n \\
                \end{array}
              \right]$, $m,n\in{\mathbb Z}$, we can write it as $\left[
                \begin{array}{c}
                 m \\
                  \frac{1+3n}{3} \\
                \end{array}
              \right]$. We now rewrite the second term in the union in $Z(\widehat{\mu})$ as ${\mathbb R}\times \{\frac{2^{j-1}(1+2n)}{3}\}$. As any integer can be written as $2^j p$, for some $j\ge 0$ and odd number $p$, this means that $\left[
                \begin{array}{c}
                 m \\
                  \frac{1+3n}{3} \\
                \end{array}
              \right]\in Z(\widehat{\mu})$, justifying the claim. As ${\mathcal Z}\neq\emptyset,$ this shows that there is no spectrum in ${\mathbb Z}^2$ for this measure.  In fact,
              $$
T(R,B) = \bigcup_{x\in K_1}\{x\}\times ([0,3]+g(x)),
$$
where $K_1$ is the Cantor set of $1/4$ contraction ratio and digit $\{0,1\}$ and $g:[0,1]\rightarrow{\mathbb R}$ is a measurable function obtaining from the off-diagonal entries.

\end{proof}
\medskip

To overcome this obstruction, as we will see, we prove that ${\mathcal Z}$ has a dynamical structure and, from this structure, we obtain that such Hadamard triples have to have a special quasi-product form, and moreover, in one of the factors, the digits form complete set of representatives.

\begin{proposition}\label{pr1.14}
Suppose that $(R,B,L)$ forms a Hadamard triple and ${\mathbb Z}[R,B] = {\mathbb Z}^d$ and let $\mu = \mu(R,B)$ be the associated self-affine measure $\mu=\mu(R,B)$. Suppose that the set

$$\mathcal Z:=\left\{x\in \br^d : \widehat\mu (x+k)=0\mbox{ for all }k\in\bz^d\right\},$$
is non-empty. Then there exists an integer matrix $M$ with $\det M=1$ such that the following assertions hold:

\begin{enumerate}
	\item The matrix $\tilde R:=MRM^{-1}$ is of the form
	\begin{equation}
	\tilde R=\begin{bmatrix} R_1&0\\ C &R_2\end{bmatrix},
	\label{eq1.14.1}
	\end{equation}
	with $R_1\in M_r(\bz)$, $R_2\in M_{d-r}(\bz)$ expansive integer matrices and  $C\in M_{(d-r)\times r}(\bz)$.
	\item If $\tilde B=MB$ and $\tilde L=(M^T)^{-1}L$, then $(\tilde R, \tilde B,\tilde L)$ is a Hadamard triple.
	\item The measure $\mu(R,B)$ is spectral with spectrum $\Lambda$ if and only if the measure $\mu(\tilde R,\tilde B)$ is spectral with spectrum $(M^T)^{-1}\Lambda$.
	\item There exists $y_0\in\br^{d-r}$ such that $(R_2^T)^my_0\equiv y_0(\mod (R_2^T)\bz^d)$ for some integer $m\geq 1$ such that the union
	$$\tilde {\mathcal S}=\bigcup_{k=0}^{m-1}(\br^r\times \{(R_2^T)^ky_0\}+\bz^d)$$
	is contained in the set
	$$\tilde{\mathcal Z}:=\left\{x\in \br^d : \widehat{\tilde \mu} (x+k)=0\mbox{ for all }k\in\bz^d\right\},$$
	where $\tilde \mu=\mu(\tilde R,\tilde B)$. The set $\tilde{\mathcal S}$ is invariant (with respect to the system $(m_{\tilde B},\tilde R^T,\tilde {\cj L},)$, see the definition below and Definition \ref{definv}) where $\tilde {\cj L}$ is a complete set of representatives $(\mod \tilde R^T\bz^d)$. In addition, all possible transitions from a point in $\br^r\times \{(R_2^T)^ky_0\}+\bz^d$, $1\leq k\leq m$ lead to a point in $\br^r\times \{(R_2^T)^{k-1}y_0\}+\bz^d$.
\end{enumerate}
\end{proposition}

The key fact in the proof of this proposition is that the set $\mathcal Z$ is invariant in the sense defined by Conze et al. in \cite{CCR}. That means that if $x\in\mathcal Z$, $k\in \bz^d$ and $m_B((R^T)^{-1}(x+k))\neq 0$ then $(R^T)^{-1}(x+k)$ is in $\mathcal Z$,  where
$$m_B(x)=\frac{1}{N}\sum_{b\in B}e^{2\pi i \ip{b}{x}}.$$
Then, the results from \cite{CCR} show that $\mathcal Z$ must have a very special form, and this implies the proposition.

\begin{theorem}\label{th_quasi}
Suppose that $$R=\begin{bmatrix}
R_1& 0\\ C& R_2
\end{bmatrix},
$$$(R,B,L_0)$ is a Hadamard triple and $\mu = \mu(R,B)$ is the associated self-affine measure and ${\mathcal Z}\neq \emptyset$. Then the set $B$ has the following quasi-product form:
\begin{equation}
B=\left\{(u_i,v_i+Qc_{i,j})^T : 1\leq i\leq N_1, 1\leq j\leq|\det R_2|\right\},
\label{eq1.18.1}
\end{equation}
where
 \begin{enumerate}\item $N_1 = N/|\det R_2|$, \item $Q$ is a $(d-r)\times (d-r)$ integer matrix with $|\det Q|\geq 2$ and $R_2Q=Q\tilde R_2$ for some $(d-r)\times(d-r)$ integer matrix $\widetilde{R_2}$, \item the set $\{Qc_{i,j}: 1\leq j\leq |\det R_2|\}$ is a complete set of representatives $(\mod R_2(\bz^{d-r}))$, for all $1\leq i\leq N_1$.
\end{enumerate}
\medskip

Moreover, one can find some $L\equiv L_0 (\mod \ R^T(\mathbb Z^d))$ so that $(R,B,L)$ is a Hadamard triple and $(R_1,\pi_1(B), L_1(\ell_2))$ and $(R_2,B_2(b_1),\pi_2(L))$ are Hadamard triples on ${\mathbb R}^r$ and ${\mathbb R}^{d-r}$ respectively for all $b_1\in\pi_1(B)$ and $l_2\in\pi_2(L)$.
\end{theorem}
Here $\pi_1,\pi_2$ are the projections onto the first and second components in $\br^d=\br^r\times\br^{d-r}$, and for $b_1\in\pi_1(B)$, $B_2(b_1):=\{b_2 : (b_1,b_2)\in B\}$ and
for $l_2\in\pi_2(L)$, $L_1(l_2):=\{l_1: (l_1,l_2)\in L\}$.

\medskip

In Example \ref{example5.3}, it is easy to see that the digit set is in a quasi-product form with $Q = 3$. Suppose now the pair $(R,B)$ is in the quasi-product form
\begin{equation}\label{R_4.1}R=\begin{bmatrix}
R_1&0\\
C&R_2
\end{bmatrix}\end{equation}
\begin{equation}
B=\left\{(u_i,d_{i,j})^T : 1\leq i\leq N_1, 1\leq j\leq N_2:=|\det R_2|\right\},
\label{eq1.23.1}
\end{equation}
and $\{d_{i,j}:1\leq j\leq N_2\}$ ($d_{i,j} = v_i +Qc_{i,j}$ as in Theorem \ref{th_quasi}) is a complete set of representatives $(\mod R_2\bz^{d-r})$.
We will show that the measure $\mu=\mu(R,B)$ has a quasi-product structure.

\medskip

Note that we have
$$R^{-1}=\begin{bmatrix}
R_1^{-1}&0\\
-R_2^{-1}CR_1^{-1}&R_2^{-1}
\end{bmatrix}$$
and, by induction,
$$R^{-k}= \begin{bmatrix}
R_1^{-k}&0\\
D_k&R_2^{-k}
\end{bmatrix},\mbox{ where }D_k:=-\sum_{l=0}^{k-1}R_2^{-(l+1)}CR_1^{-(k-l)}.$$
 For the invariant set $T(R,B)$, we can express it as a set of infinite sums,
 $$
 T(R,B)=\left\{\sum_{k=1}^\infty R^{-k}b_k : b_k\in B\right\}.
 $$
Therefore any element $(x,y)^T\in T(R,B)$ can be written in the following form
$$x=\sum_{k=1}^\infty R_1^{-k}a_{i_k}, \quad y=\sum_{k=1}^\infty D_ka_{i_k}+\sum_{k=1}^\infty R_2^{-k}d_{i_k,j_k}.$$
Let $X_1$ be the attractor (in $\br^r)$ associated to the IFS defined by the pair $(R_1,\pi_1(B)=\{u_i :1\leq i\leq N_1\})$ (i.e. $X_1=T(R_1,\pi_1(B))$). Let $\mu_1$ be the (equal-weight) invariant measure associated to this pair.

\medskip

For each sequence $\omega=(i_1i_2\dots)\in\{1,\dots,N_1\}^{\bn} = \{1,\dots, N_1\}\times\{1,\dots, N_1\}\times...$, define
\begin{equation}\label{eqxomega}
x(\omega)=\sum_{k=1}^\infty R_1^{-k}u_{i_k}.
\end{equation}
 As $(R_1,\pi_1(B))$ forms Hadamard triple with some $L_1(\ell_2)$, the measure $\mu(R_1,\pi_1(B))$ has the no-overlap property. It implies that for $\mu_1$-a.e. $x\in X_1$, there is a unique $\omega$ such that $x(\omega)=x$. We define this as $\omega(x)$. This establishes a bijective correspondence, up to measure zero, between the set $\Omega_1:=\{1,\dots,N_1\}^{\bn}$ and $X_1$. The measure $\mu_1$ on $X_1$ is the pull-back of the product measure which assigns equal probabilities $\frac1{N_1}$ to each digit.

\medskip

For $\omega=(i_1i_2\dots)$ in $\Omega_1$, define
$$\Omega_2(\omega):=\{(d_{i_1,j_1}d_{i_2,j_2}\dots d_{i_n,j_n}\dots) : j_k\in \{1,\dots,N_2\}\mbox{ for all }k\in\bn\}.$$
For $\omega\in\Omega_1$, define $g(\omega):=\sum_{k=1}^\infty D_ka_{i_k}$ and $g(x):=g(\omega(x))$, for $x\in X_1$.  Also $\Omega_2(x):=\Omega_2(\omega(x))$.
For $x\in X_1$, define
$$X_2(x):=X_2(\omega(x)):=\left\{\sum_{k=1}^\infty R_2^{-k}d_{i_k,j_k}: j_k\in\{1,\dots,N_2\}\mbox{ for all }k\in\bn\right\}.$$
Note that the attractor $T(R,B)$ has the following form

$$T(R,B)=\{(x,g(x)+y)^T: x\in X_1,y\in X_2(x)\}.$$

\medskip

For $\omega\in \Omega_1$,  consider the product probability measure $\mu_\omega$, on $\Omega_2(\omega)$, which assigns equal probabilities $\frac{1}{N_2}$ to each digit $d_{i_k,j_k}$ at level $k$.
Next, we define the measure $\mu_\omega^2$ on $X_2(\omega)$. Let
$r_\omega:\Omega_2(\omega)\rightarrow X_2(\omega)$,
$$r_\omega(d_{i_1,j_1}d_{i_2,j_2}\dots)=\sum_{k=1}^\infty R_2^{-k}d_{i_k,j_k}.$$
Define $\mu_x^2:=\mu_{\omega(x)}^2:=\mu_{\omega(x)}\circ r_{\omega(x)}^{-1}$.

\medskip

Note that the measure $\mu_x^2$ is the infinite convolution product $\delta_{R_2^{-1}B_2(i_1)}\ast\delta_{R_2^{-2}B_2(i_2)}\ast\dots$, where $\omega(x)=(i_1i_2\dots)$, $B_2(i_k):=\{d_{i_k,j} : 1\leq j\leq N_2\}$ and $\delta_A:=\frac{1}{\#A}\sum_{a\in A}\delta_a$, for a subset $A$ of $\br^{d-r}$.

The following lemmas were proved in \cite{DJ07d}.

\begin{lemma}\label{lem1.23}\cite[Lemma 4.4]{DJ07d}
For any bounded Borel functions on $\br^d$,
$$\int_{T(R,B)}f\,d\mu=\int_{X_1}\int_{X_2(x)}f(x,y+g(x))\,d\mu_x^2(y)\,d\mu_1(x).$$
\end{lemma}

\begin{lemma}\label{lem1.24}\cite[Lemma 4.5]{DJ07d}
If $\Lambda_1$ is a spectrum for the measure $\mu_1$, then
$$F(y):=\sum_{\lambda_1\in\Lambda_1}|\widehat\mu(x+\lambda_1,y)|^2=\int_{X_1}|\widehat\mu_s^2(y)|^2\,d\mu_1(s),\quad(x\in\br^r,y\in\br^{d-r}).$$
\end{lemma}

The two lemmas lead to the following proposition.

\begin{proposition}\cite{DHL15}\label{lem1.25}
For the quasi-product form given in (\ref{R_4.1}) and (\ref{eq1.23.1}), there exists a lattice $\Gamma_2$ such that for $\mu_1$-almost every $x\in X_1$, the set $\Gamma_2$ is a spectrum for the measure $\mu_x^2$.
\end{proposition}

Finally, the proof of the Theorem \ref{th1} follows by induction on the dimension $d$: we know it is true for $d=1$ from \cite{DJ06}. Then assume it is true for dimensions up to $d-1$. The case $\mathcal Z=\ty$ was treated before; if $\mathcal Z\neq\ty$, then the measure $\mu$ is in the quasi-product described above. With Proposition \ref{lem1.25}, and using induction, the measure $\mu_1$ has a spectrum $\Gamma_1$ and then the measure $\mu$ has the spectrum $\Gamma_1\times\Gamma_2$.

\medskip
\section{Non-spectral singular measures with Fourier frames}

Suppose that instead of the Hadamard triple, we are given the almost-Parseval-frame tower in Definition \ref{APFT}. A similar approach in Proposition \ref{prop_main} (for details see \cite{LaiW})  allows us to prove the following:

\begin{proposition}
Suppose that $\{(N_j,B_j)\}$ is an {\it almost-Parseval-frame tower} associated to $\{\epsilon_j\}$ with $\{L_j\}$ as its pre-spectrum. Let $
\Lambda = \bigcup_{n=1}^{\infty}\Lambda_n $ where $\Lambda_n = L_1+N_1L_2+...+(N_1...N_{n-1})L_n$ and let $\mu$ be the measure defined in (\ref{measure}). Assume that
$$
\delta(\Lambda): = \inf_{n\ge1}\inf_{\lambda\in\Lambda_n}|\widehat{\mu_{>n}}(\lambda)|^2>0
$$
Then $\Lambda$ is a frame spectrum for $L^2(\mu)$ and for any $f\in L^2(\mu)$,
\begin{equation}\label{eq3.1i}
\delta(\Lambda)\prod_{j=1}^{\infty}(1-\epsilon_j)^2\|f\|^2\leq \sum_{\lambda\in\Lambda}\left|\int f(x)e^{-2\pi i \langle\lambda,x\rangle}d\mu(x)\right|^2\leq \prod_{j=1}^{\infty}(1+\epsilon_j)^2\|f\|^2.
\end{equation}
\end{proposition}

\medskip

If we have a self-similar almost-Parseval-frame tower (all $N_j = N$), then the measure $\mu_{>n}(\cdot)$ can be written as $\mu(N^n\cdot)$. In this case, we can produce new candidates of frame spectra as in (\ref{eqLambda_k}) and (\ref{eqlambda}) and the consideration for ${\mathcal Z}= \emptyset$ works in a similar way as in Proposition \ref{prop_main2}. We have

\begin{theorem} \label{theorem 3.1}
Suppose that $(N_j,B_j)$ is a self-similar almost-Parseval-frame tower and the associated measure $\mu$ satisfies ${\mathcal Z} = \emptyset$. Then $\mu$ is a frame-spectral measure.
\end{theorem}

 For a self-similar measure on ${\mathbb R}^1$ as defined in Definition \ref{APFT}, ${\mathcal Z} = \emptyset$ can be obtained without additional assumption.

  \medskip

 \noindent{\it Proof of Theorem \ref{th2} (i).} By Theorem \ref{theorem 3.1}, it suffices to show that ${\mathcal Z} = \emptyset$ for self-similar measures $\mu(N,B)$ defined by the almost-Parseval-frame tower in Definition \ref{APFT}. In fact, as $B\subset \{0,1,...,N-1\}$, the self-similar set $T(N,B)$ is a compact set inside $[0,1]$. By the Stone-Weierstrass theorem, the linear span of exponentials $e_n$ with integer frequencies is complete in the space of continuous functions on $T(N,B)$. This shows that ${\mathcal Z} = \emptyset$, completing the proof. \qquad$\Box$

\medskip

In the end of this section, we demonstrate the existence of an almost-Parseval-frame tower with $\epsilon_j>0$, which gives a proof of Theorem \ref{th2}(ii). More precisely, we prove

\begin{theorem}\label{th0.1} Let $N_n$ and $M_n$ be positive integers satisfying
\begin{equation}\label{eq0.2}
N_j = M_jK_j+\alpha_j
\end{equation}
for some integer $K_j$ and $0\leq\alpha_j<M_j$ with
\begin{equation}\label{eq0.3}
\sum_{j=1}^{\infty}\frac{\alpha_j\sqrt{M_j}}{K_j}<\infty.
\end{equation}
Define
\begin{equation}\label{eq0.4}
B_j = \{0, K_j,...,(M_j-1)K_j\}, \ L_j = \{0,1,...,M_j-1\}.
\end{equation}
Then   $(N_{j}, B_{j})$ forms an almost-Parseval-frame tower associated with
$$
\epsilon_j=\frac{2\pi\alpha_j\sqrt{M_j}}{{K_j}}
$$ and its pre-spectrum is $\{L_j\}$. 
\end{theorem}

\begin{proof}

Let
$$
{\mathcal F}_j= \frac{1}{\sqrt{M_j}}\left[e^{2\pi i b \lambda/N_j}\right]_{\lambda\in L_j,b\in B_j}, \ {\mathcal H}_j= \frac{1}{\sqrt{M_j}}\left[e^{2\pi i b \lambda/M_jK_j}\right]_{\lambda\in L_j,b\in B_j}.
$$
Then ${\mathcal H}_j$ is a unitary matrix (in fact Hadamard matrices). We first show that for any $j>0$, the operator norm ($\|A\|: = \max_{\|x\|=1}\|Ax\|$)
\begin{equation}\label{eq2.0}
\|{\mathcal F}_j-{\mathcal H}_j\|\leq\frac{2\pi \alpha_j\sqrt{M_j}}{K_j}.
\end{equation}
To see this, We note that by Cauchy-Schwarz inequality,
\begin{equation}\label{eq2.1}
\|{\mathcal F}_j-{\mathcal H}_j\|^2\leq  \frac{1}{M_j}\sum_{b\in B_j}\sum_{\lambda\in L_j}\left|e^{2\pi i b\lambda/N_j}-e^{2\pi i b\lambda/M_jK_j}\right|^2.
\end{equation}
We now estimate the difference of the exponentials inside the summation using an elementary estimate
$$
|e^{i\theta_1}-e^{i\theta_2}| = |e^{i(\theta_1-\theta_2)}-1|\leq |\theta_1-\theta_2|.
$$
This implies that
$$
\begin{aligned}
\left|e^{2\pi i b\lambda/N_j}-e^{2\pi i b\lambda/M_jK_j}\right|^2\leq& \left|\frac{2\pi b\lambda}{N_j}-\frac{2\pi b\lambda}{M_jK_j}\right|^2\nonumber\\
=& 4\pi^2\frac{b^2\lambda^2 \alpha_j^2}{M_j^2K_j^2N_j^2}  \ \ \  \ \mbox{(by $N_j = M_jK_j+\alpha_j$)}\nonumber\\
\leq &4\pi^2\frac{M_j^2\alpha_j^2}{N_j^2}  \ \ \ \ \mbox{(by $b\leq M_jK_j$ and $\lambda\leq M_j$)}\\
\end{aligned}
$$
Hence, from (\ref{eq2.1}),
\begin{eqnarray}
\|{\mathcal F}_j-{\mathcal H}_j\|^2&\leq& \frac{1}{M_j}\sum_{b\in B_j}\sum_{\lambda\in L_j}4\pi^2\frac{M_j^2\alpha_j^2}{N_j^2} \nonumber\\
 &=& 4\pi^2\frac{M_j^3\alpha_j^2}{N_j^2} \nonumber \\
 &= & 4\pi^2\frac{M_j\alpha_j^2}{\left( K_j+\alpha_j/M_j\right)^2}
\end{eqnarray}
As $\alpha_j\ge 0$, $\|{\mathcal F}_j-{\mathcal H}_j\|^2\leq  4\pi^2{\alpha_j^2 M_j}/{K_j^2}$ and thus  (\ref{eq2.0}) follows by taking square root.

\medskip

 We now show that $\{(N_{j}, B_{j})\}$ forms an almost-Parseval-frame tower with pre-spectrum $L_j$. The first two conditions for the almost-Parseval-frame tower are clearly satisfied. To see the last condition, we recall that $\epsilon_j = 2\pi \sqrt{M_j}\alpha_j/{K_j} $.  From the triangle inequality and (\ref{eq2.0}), we have
$$
\begin{aligned}
\|{\mathcal F}_j{\bf w}\|\leq&\|{\mathcal H}_j{\bf w}\| + \|{\mathcal F}_j-{\mathcal H}_j\|\|{\bf w}\|\\
\leq&  \left(1+\frac{2\pi \alpha_j\sqrt{M_j}}{K_j}\right)\|{\bf w}\| = (1+\epsilon_j)\|{\bf w}\|.\\
\end{aligned}
$$
Similarly, for the lower bound,
$$
\begin{aligned}
\|{\mathcal F}_{j}{\bf w}\|\geq &\|{\mathcal H}_{j}{\bf w}\|-\|{\mathcal F}_{j}-{\mathcal H}_{j}\|\|{\bf w}\|\\
\geq& \left(1-\frac{2\pi \alpha_j\sqrt{M_j}}{K_j}\right)\|{\bf w}\| = (1-\epsilon_j)\|{\bf w}\|.
\end{aligned}
$$
Thus, from (\ref{ALmost1}), the last condition follows and $(N_j,B_j)$ satisfies the almost-Parseval-frame condition associated with $\{\epsilon_j\}$ and $\sum_{j=1}^{\infty}\epsilon_j<\infty$ is guaranteed by (\ref{eq0.3}) in the assumption.
\end{proof}

Some of the fractal measures induced by the almost-Parseval-frame tower were found to be non-spectral, one can refer to \cite{LaiW} for detail.

\medskip

\section{Explicit construction of spectrum}
In general, the canonical orthogonal set in Proposition \ref{prop_main} is not necessarily a spectrum. But in some cases, we can complete this set by adding some more points, and, in some cases, it is possible to gives an explicit formula for the spectrum of the measure $\mu(R,B)$. Such a description can be given in the following definition. This is always true in dimension one, as explained in \cite{DJ06}.
\begin{definition}\label{definv}
Let $(R,B,L)$ be a Hadamard triple. We define the function
$$m_B(x)=\frac{1}{\#B}\sum_{b\in B}e^{2\pi i \ip{b}{x}},\quad (\xi\in\br^d).$$
The Hadamard triple condition implies that $\delta_{R^{-1}B}$ is a spectral measure with a spectrum $L$ and $m_{R^{-1}B}$ is the Fourier transform of this measure.  The Hadamard condition implies that
\begin{equation}\label{M_B}
\sum_{\ell\in L}|m_B((R^T)^{-1}(x+\ell))|^2=1, \mbox{or} \ \sum_{\ell\in L}|m_B(\tau_{\ell}(x))|^2=1,
\end{equation}
where we define the maps
$$\tau_{\ell}(x)=(R^T)^{-1}(x+\ell),\quad(x\in\br^d,\ell\in L), \ \mbox{and} \ \tau_{\ell_1...\ell_m} = \tau_{\ell_1}\circ...\circ\tau_{\ell_m}.
$$
A closed set $K$ in $\br^d$ is called {\it invariant (with respect to the system $(R,B,L)$)} if, for all $x\in K$ and all $\ell\in L$
$$
 |m_B(\tau_{\ell}(x))|>0 \ \Longrightarrow \ \tau_{\ell}(x)\in K.
 $$
We say that {\it the transition, using $\ell$, from $x$ to $\tau_\ell( x)$ is possible}, if $\ell\in  L$ and $m_B(\tau_{\ell} (x))>0$. A compact invariant set is called {\it minimal} if it does not contain any proper compact invariant subset.

For $\ell_1,\dots,\ell_m\in L$, the cycle $\mathcal C(\ell_1,\dots,\ell_m)$ is the set
$$\mathcal C(\ell_1,\dots,\ell_m)=\{x_0,\tau_{\ell_m}(x_0),\tau_{\ell_{m-1}\ell_m}(x_0),\dots, \tau_{\ell_2\dots\ell_m}(x_0)\},$$
where $x_0:=\wp(\ell_1,\dots, \ell_m)$ is the fixed point of the map $\tau_{\ell_1...\ell_m}$. i.e. $\tau_{\ell_1...{\ell_m}}(x_0)=x_0$.
The cycle $\mathcal C(\ell_1,\dots, \ell_m)$ is called an {\it extreme cycle for $(R,B,L)$} if $|m_B(x)|=1$ for all $x\in \mathcal C(\ell_1,\dots, \ell_m)$.
\end{definition}

\begin{definition}\label{definv2}
We say that the Hadamard triple $(R,B,L)$ is {\it dynamically simple} if the only minimal compact invariant set are extreme cycles. For a Hadamard triple $(R,B,L)$, the {\it orthonormal set $\Lambda$ generated by extreme cycles} is the smallest set such that
 \begin{enumerate}
 \item it contains $-{\mathcal C}$ for all extreme cycles ${\mathcal C}$ for $(R,B,L)$
 \item it satisfies $R^T\Lambda+L\subset \Lambda$.
 \end{enumerate}
When this set $\Lambda$ is a spectrum (see Theorem \ref{th_dynamic} below), we call it {\it the dynamically simple spectrum}.

More generally, the {\it set generated by an invariant subset $A$} of $\br^d$, is the smallest set which contains $-A$ and satisfies (ii).
\end{definition}

\medskip

\begin{theorem}\label{th_dynamic}\cite{DL15c}
Let  $(R,B,L)$ be a dynamically simple Hadamard triple. Then the orthonormal set $\Lambda$ generated by extreme cycles is a spectrum for the self-affine measure $\mu_{R,B}$ and $\Lambda$ is explicitly given by
$$\Lambda=\{\ell_0+R^T\ell_1+\dots (R^T)^{n-1}\ell_{n-1}+(R^T)^n(-c) : \ell_0,\dots,\ell_{n-1}\in L, n\geq0, \ c\mbox{ are extreme cycle points}\}.
$$
Moreover, if $(R,B,L)$ is a Hadamard triple on ${\mathbb R}^1$, it must be dynamically simple.
\end{theorem}

\begin{example}
There are Hadamard triples which are not dynamically-simple. For example let
$$R=\begin{bmatrix}2&1\\
0&2\end{bmatrix},\quad B=\left\{\begin{bmatrix}0\\0\end{bmatrix},\begin{bmatrix}3\\0\end{bmatrix},\begin{bmatrix}0\\1\end{bmatrix},\begin{bmatrix}3\\1\end{bmatrix}\right\},
L=\left\{\begin{bmatrix}0\\0\end{bmatrix},\begin{bmatrix}1\\0\end{bmatrix},\begin{bmatrix}0\\1\end{bmatrix},\begin{bmatrix}1\\1\end{bmatrix}\right\}.$$

\begin{figure}[h]
\centering
\begin{minipage}{.5\textwidth}
\centering
\includegraphics[width=\linewidth]{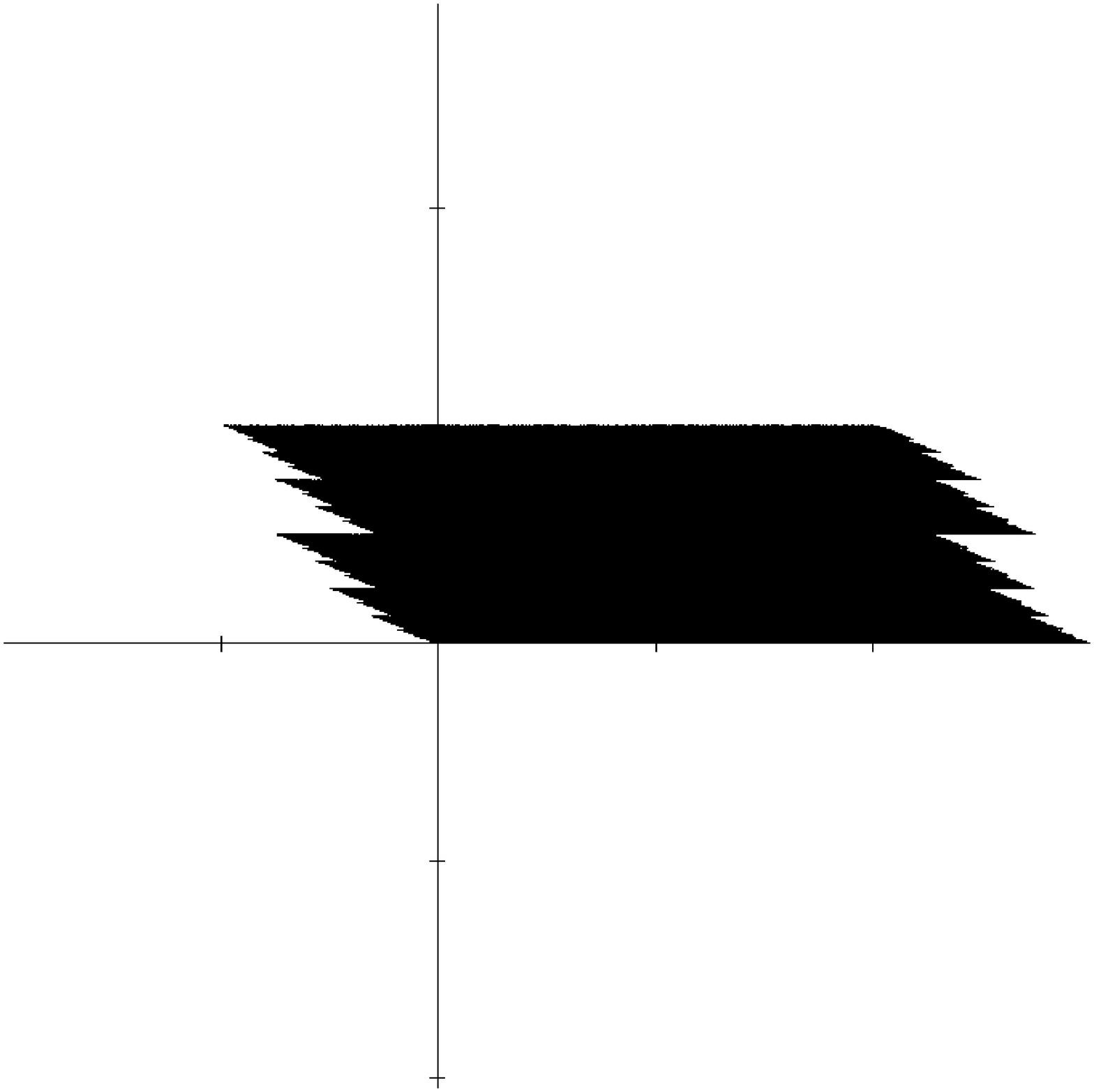}%
\caption{$T(R,B)$}
\label{a}
\end{minipage}%
\begin{minipage}{.5\textwidth}
\centering
\includegraphics[width=\linewidth]{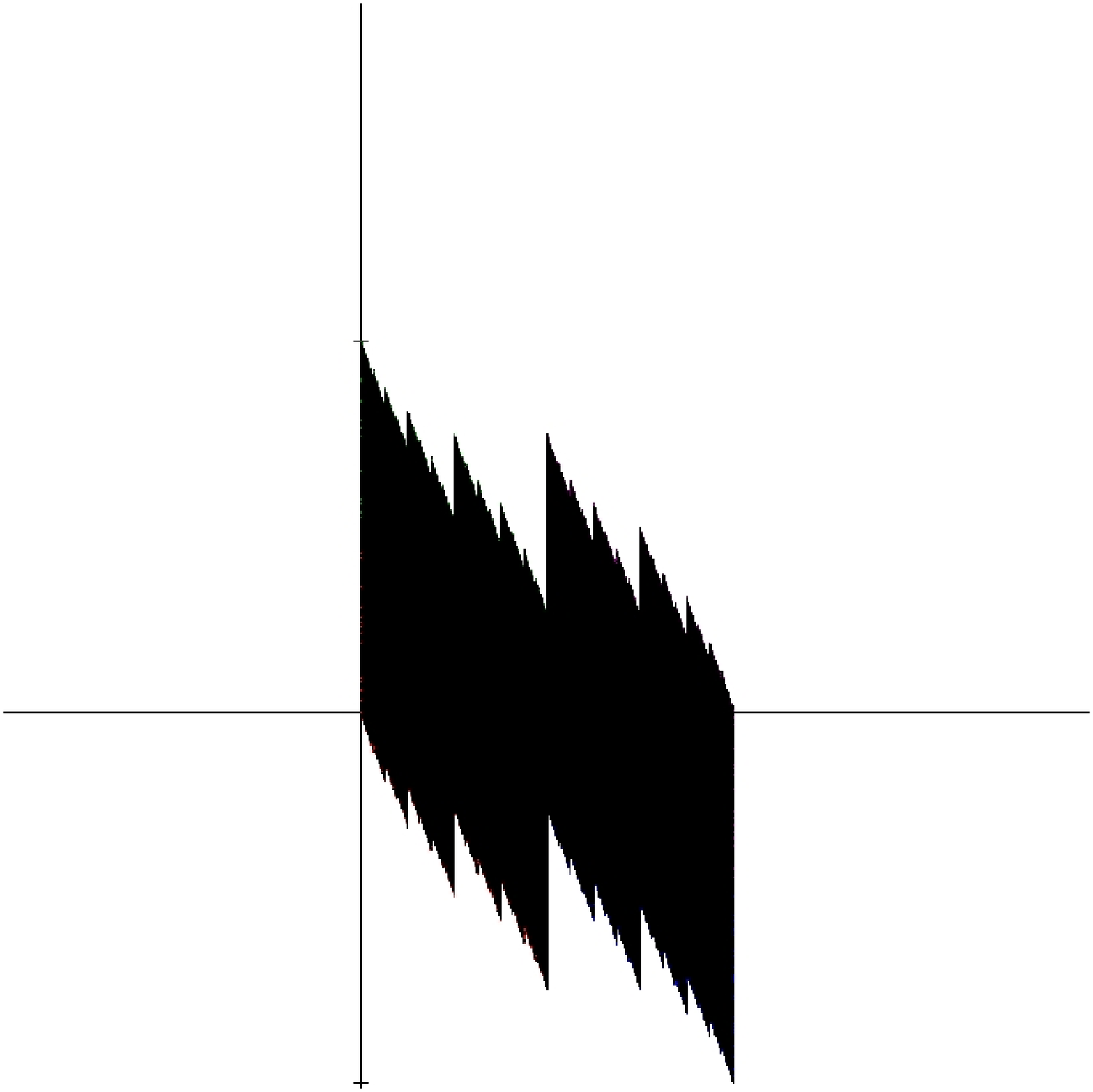}%
\caption{$T(R^T,L)$}
\label{b}
\end{minipage}
\end{figure}

Note that $B$ is a complete set of representatives $\mod R\bz^2$ and $L$ is a complete set of representatives $\mod R^T\bz^2$; thus we have a Hadamard triple. As shown in \cite[Example 2.3]{LW1}, the measure $\mu(R,B)$ is the normalized Lebesgue measure on the attractor $T(R,B)$ which tiles $\br^2$ with $3\bz\times\bz$. Hence $\frac13\bz\times\bz$ is a spectrum for $\mu(R,B)$ (by \cite{Fug74}).

We look for the extreme cycles: we have
$$m_B(x,y)=\frac{1}{4}(1+e^{2\pi i3x}+e^{2\pi iy}+e^{2\pi i(3x+y)}).$$
If we want $|m_B(x,y)|=1$, then we must have equality in the triangle inequality and we get that $3x,y,(3x+y)\in\bz$. So $(x,y)\in \frac13\bz\times\bz$.
For an extreme cycle $(x,y)$, we must also have that $(x,y)$ is in the attractor of the IFS $(R^T,L)$, and this is contained in $[0,1]\times [-1,1]$ (on can check the invariance of this rectangle for the IFS). So, one can check which of the points in $\frac13\bz\times\bz\cap [0,1]\times[-1,1]$ are in an extreme cycle. The only extreme cycles are $\left\{\begin{bmatrix}0\\0\end{bmatrix}\right\}$,$\left\{\begin{bmatrix}1\\0\end{bmatrix}\right\}$,$\left\{\begin{bmatrix}1\\0\end{bmatrix}\right\}$,$\left\{\begin{bmatrix}0\\1\end{bmatrix}\right\}$,$\left\{\begin{bmatrix}1\\-1\end{bmatrix}\right\}$. Note that these form a complete set of representatives $\mod R^T\bz^2$. The set generated by the extreme cycles is then just $\bz^2$ which is a proper subset of the spectrum $\frac13\bz\times\bz$, so the Hadamard triple is not dynamically simple.
\end{example}

\medskip

\section{Open problems}

The major open problem in the study of Fourier analysis on fractals is to see whether the non-spectral self-affine measures are still frame-spectral. The idea of almost-Parseval-frame towers turns this problem into a problem of matrix analysis. Given an integral expanding matrix $R$ and a set of simple digits $B$ with $N=\#B<|\det R|$, the condition of  almost-Parseval-frame towers can be reformulated equivalently as {\it for any $\epsilon>0$, there exists $n\in{\mathbb N}$ and a set of $L_n\subset {\mathbb Z}^d$ such that the matrix}
$$
{\mathcal F}_n (B_n,L_n) = \frac{1}{\sqrt{N^n}}\left(e^{2\pi i \langle R^{-n} b,\ell\rangle}\right)_{\ell\in L_n,b\in B_n}
$$
{\it satisfies}
$$
(1-\epsilon)\|{\bf w}\|^2\le \|{\mathcal F}_n{\bf w}\|^2\le(1+\epsilon)\|{\bf w}\|^2
$$
{\it for any vectors ${\bf w}\in {\mathbb C}^{N^n}$. (Recall that $B_n = B+RB+...+R^{n-1}B$)}

\medskip

We observe that if we let ${\overline{B_n}}$ and $\overline{L_n}$ be respectively the complete representative class  (mod $R^n({\mathbb Z}^d))$ and  (mod $(R^T)^n({\mathbb Z}^d))$. Then the matrix
$$
\overline{{\mathcal F}_n}: = \frac{1}{\sqrt{|\det R|^n}}\left(e^{2\pi i \langle R^{-n} b,\ell\rangle}\right)_{\ell\in \overline{L_n},b\in \overline{B_n}}
$$
forms a unitary matrix. i.e.
$$
\|\overline{{\mathcal F}_n}{\bf w}\| = \|{\bf w}\|, \ \forall {\bf w}\in{\mathbb C}^{|\det R|^n}
$$
 As $B_n\subset {\overline{B_n}}$, we can take the vectors ${\bf w}$ such that they are zero on the coordinates which are not in $B_n$. This implies that
 $$
\| {\mathcal F}_n(B_n,\overline{L_n}){\bf w}\| = \frac{|\det R|^n}{N^n}\|{\bf w}\|.
 $$
 In other words,
 $$
 \sum_{\lambda\in \overline{L_n}}\left|\sum_{b\in B_n} w_b \frac{1}{\sqrt{N^n}}e^{-2\pi i \langle R^{-n}b,\lambda\rangle}\right|^2 = \frac{|\det R|^n}{N^n}\sum_{b\in B_n}|w_b|^2.
 $$
 This shows that the collection of vectors $\{ \left(\frac{1}{\sqrt{N^n}}e^{-2\pi i \langle R^{-n}b,\lambda\rangle}\right)_{b\in  B_n}:\lambda\in \overline{L_n}\}$ forms a tight frame for ${\mathbb C}^{N^n}$ with frame bound $\frac{|\det R|^n}{N^n}$. Our problem is to extract a subset $L_n$ from $\overline{L_n}$ such that we have an almost tight frame with frame constant nearly 1. This reminds us about the Kadison-Singer problem that was open for over 50 years and solved recently in \cite{MSS}.

 \medskip

\begin{theorem}\cite[Corollary 1.5]{MSS}
Let $r$ be a positive integer and let $u_1,...,u_m\in{\mathbb C}^d$ such that
$$
\sum_{i=1}^m|\langle w, u_i\rangle|^2 = \|w\|^2 \ \forall w\in {\mathbb C}^d
$$
and $\|u_i\|\leq \delta$ for all $i$. Then there exists a partition $S_1,...,S_r$ of $\{1,...,m\}$ such that
$$
\sum_{i\in S_j}|\langle w, u_i\rangle|^2 \le  \left(\frac{1}{\sqrt{r}}+\sqrt{\delta}\right)^2\|w\|^2 \ \forall w\in {\mathbb C}^d.
$$
 \end{theorem}

 \medskip

 This statement says that we can partition a tight frame into $r$ subsets such that the frame constant of each partition is almost $1/r$. Iterating this process allowed Nitzan et al \cite{NOU} to establish the existence of Fourier frames on any unbounded sets of finite measure. One of their lemmas states:

\begin{lemma}\label{lem6.1}
\cite[Lemma 3]{NOU} Let $A$ be an $K\times L$ matrix and $J\subset \{1,...,K\}$, we denote by $A(J)$ the sub-matrix of $A$ whose rows belong to the index $J$. Then there exist universal constants $c_0,C_0>0$ such that whenever $A$ is a $K\times L$ matrix, which is a sub-matrix of some $K\times K$ orthonormal matrix, such that all of its rows have equal $\ell^2$-norm, one can find a subset $J\subset\{1,...,K\}$ such that
$$
c_0\frac{L}{K}\|{\bf w}\|^2\leq \|A(J){\bf w}\|^2\leq C_0\frac{L}{K}\|{\bf w}\|^2 ,\ \forall {\bf w}\in {\mathbb C}^n.
$$
\end{lemma}

\medskip
This lemma leads naturally to the following:

\begin{proposition}\label{prop1.3}
With $(R,B)$ as in Definition \ref{defifs}, there exist  universal constants $0< c_0< C_0<\infty$ such that for all $n$, there exists $J_n$ such that
$$
c_0\sum_{b\in B_n}|w_b|^2\leq \sum_{\lambda\in J_n}\left|\frac{1}{\sqrt{N^n}}\sum_{b\in B_n}w_be^{-2\pi i \langle R^{-n}b, \lambda\rangle}\right|^2\leq C_0\sum_{b\in B_n}|w_b|^2
$$
for all $(w_b)_{b\in B_n}\in{\mathbb C}^{N^n}$.
\end{proposition}

\begin{proof} Let
$$
{\mathcal F}_n = \frac{1}{|\det R|^{n/2}}\left[e^{2\pi i \langle R^{-n}b,\ell\rangle}\right]_{\ell\in \overline{L}_n,b\in\overline{B}_n}
$$
where $\overline{B}_n$ is a complete coset representative (mod $R({\mathbb Z}^d)$) containing $B_n$ and  $\overline{L}_n$ is a complete coset representative (mod $R^T({\mathbb Z}^d)$). It is well known that ${\mathcal F}_n$ is an orthonormal matrix. Let $K = |\det R|^n$ and
$$
A_n = \frac{1}{|\det R|^{n/2}}\left[e^{2\pi i \langle R^{-n}b,\ell\rangle}\right]_{\ell\in \overline{L}_n,b\in B_n}.
$$
Then $A_n$ is a sub-matrix of ${\mathcal F}_n$ whose columns are exactly the ones with index in $B_n$ so that the size $L$ is $L = N^n$. By Lemma \ref{lem6.1}, we can find universal constants $c_0,C_0$, independent of $n$, such that for some $J_n\subset \overline{L}_n$, we have
$$
c_0\frac{N^n}{|\det R|^n}\|{\bf w}\|^2\leq \|A(J_n){\bf w}\|^2\leq C_0\frac{N^n}{|\det R|^n}\|{\bf w}\|^2 ,\ \forall {\bf w}\in {\mathbb C}^{N^n}.
$$
As $\frac{|\det R|^{n/2}}{N^{n/2}}A(J_n) = \frac{1}{|\det R|^{n/2}}\left[e^{2\pi i \langle R^{-n}b,\ell\rangle}\right]_{\ell\in J_n,b\in B_n}: = F_n$, this shows
$$
c_0\|{\bf w}\|^2\leq \|F_n{\bf w}\|^2\leq C_0\|{\bf w}\|^2 ,\ \forall {\bf w}\in {\mathbb C}^{N^n}.
$$
This is equivalent to the inequality we stated.
\end{proof}

This proposition shows that there always exists some subsets $J_n$ in which the norm  of ${\mathcal F}(B_n,J_n)$ is uniformly bounded by universal constants $c_0, C_0$, this indicates that the existence of almost-Parseval-frame pairs $B_n,L_n$ is possible.

 \begin{acknowledgements}
This work was partially supported by a grant from the Simons Foundation (\#228539 to Dorin Dutkay).
\end{acknowledgements}

\bibliographystyle{alpha}	
\bibliography{eframes}

\end{document}